\documentclass{amsart}

\usepackage{amssymb}
\usepackage{amsthm}
\usepackage{amsmath}
\usepackage{enumerate}
\usepackage{mathrsfs}

\theoremstyle{plain}

\newtheorem{thm}{Theorem}[section]

\newtheorem{lem}[thm]{Lemma}

\newtheorem{prop}[thm]{Proposition}

\def\spl#1#2{\mathord{\mathrm{Sp}_{#1}(#2)}}

\begin{document}

\title[A lower bound for the number of classes]{A lower bound for the number of conjugacy classes of a finite group}

\author{Attila Mar\'oti} \address{Fachbereich
Mathematik, Technische Universit\"{a}t Kaiserslautern, Postfach
3049, 67653 Kaiserslautern, Germany \and Alfr\'ed R\'enyi Institute of
  Mathematics, Re\'altanoda utca 13-15, H-1053, Budapest, Hungary}
\email{maroti@mathematik.uni-kl.de \and maroti.attila@renyi.mta.hu}

\thanks{The research of the author was supported by an Alexander von Humboldt
Fellowship for Experienced Researchers and by OTKA
K84233.}

\begin{abstract}
Every finite group whose order is divisible by a prime $p$ has at least $2 \sqrt{p-1}$ conjugacy classes.  
\end{abstract}
\maketitle

\section{Introduction}

Let $k(G)$ denote the number of conjugacy classes of a finite group $G$. This is also the number of complex irreducible characters of $G$. Bounding $k(G)$ is a fundamental problem in group and representation theory. 

Concerning Problem 3 of the list of 40 problems of R. Brauer \cite{Brauer}, the best known asymptotic general lower bounds for $k(G)$ in terms of the order of $G$ are almost logarithmic and are due to L. Pyber \cite{Pyber} and T. M. Keller \cite{Keller}. In this paper we consider a slightly different point of view in establishing lower bounds for $k(G)$. We wish to give a lower bound for $k(G)$ in terms of a prime divisor $p$ of $|G|$. This is related to Problem 21 in \cite{Brauer}. 

Let $G$ be a finite group such that the order of $G$ contains a prime $p$ with exact exponent $1$. Pyber observed that results of Brauer \cite{Br} imply that $G$ contains at least $2 \sqrt{p-1}$ conjugacy classes. Motivated by this observation Pyber asked various questions concerning lower bounds for $k(G)$ in terms of the prime divisors of $|G|$. In response to these questions (and motivated by trying to find explicit lower bounds for the number of complex irreducible characters in a block) L. H\'ethelyi and B. K\"ulshammer obtained various results \cite{HK}, \cite{HK1} for solvable groups. For example they proved in \cite{HK} that every solvable finite group $G$ whose order is divisible by $p$ has at least $2 \sqrt{p-1}$ conjugacy classes. Later G. Malle \cite[Section 2]{M} showed that if $G$ is a minimal counterexample to the inequality $k(G) \geq 2 \sqrt{p-1}$ with $p$ dividing $|G|$ then $G$ has the form $HV$ where $V$ is an irreducible faithful $H$-module for a finite group $H$ with $(|H|,|V|) = 1$ where $p$ is the prime dividing $|V|$. He also showed that $H$ cannot be an almost quasisimple group. Using these results, Keller \cite{K} showed that there exists a universal constant $C$ so that whenever $p > C$ then $k(G) \geq 2 \sqrt{p-1}$. In a later paper H\'ethelyi, E. Horv\'ath, Keller and A. Mar\'oti \cite{HHKM} proved that by disregarding at most finitely many non-solvable $p$-solvable
groups $G$, we have $k(G) \geq 2 \sqrt{p-1}$ with equality if and only if $\sqrt{p-1}$ is an integer, $G = C_{p} \rtimes
C_{\sqrt{p-1}}$ and $C_{G}(C_{p}) = C_{p}$. However since the constant $C$ in Keller's theorem was unspecified, there had been no quantitative information on what was meant by at most finitely many in the afore-mentioned theorem. 

In this paper we answer this question for all primes $p$.

\begin{thm}
\label{main}
Every finite group $G$ whose order is divisible by a prime $p$ has at least $2\sqrt{p-1}$ conjugacy classes. Equality occurs if and only if $\sqrt{p-1}$ is an integer, $G = C_{p} \rtimes C_{\sqrt{p-1}}$ and $C_{G}(C_{p}) = C_{p}$.
\end{thm}

\section{A reduction}

Let $G$ be a minimal counterexample to the statement of Theorem \ref{main}. By \cite{HK} (and the equality by \cite[Theorem 2.1]{HHKM}) we know that $G$ is not solvable. Also, by \cite[Theorem 3.1]{HHKM}, we may assume that $G$ is a $p$-solvable group (whose order is divisible by $p$). Now we may proceed as in \cite[Page 428]{HHKM}. Let $V$ be a minimal normal subgroup in $G$. If $|G/V|$ is divisible by $p$ then, by the minimality of $G$, we have $k(G) > k(G/V) \geq 2 \sqrt{p-1}$, a contradiction. So $p$ divides $|V|$, and since $G$ is $p$-solvable, we see that $V$ is an elementary abelian $p$-group. By this argument we see that $V$ is the unique minimal normal subgroup of $G$. By the Schur-Zassenhaus theorem, there is a complement $H$ of $V$ in $G$. So $G$ has the form $HV$ where $V$ is a coprime, faithful and irreducible $H$-module. 

In the papers \cite{VL} and \cite{VS} all non-nilpotent finite groups are classified with at most $14$ conjugacy classes. By going through these lists of groups we see that no group $G$ of the form described in the previous paragraph is a counterexample to Theorem \ref{main}. So we have $k(G) \geq 15$. This means that we can assume that $2\sqrt{p-1} \geq 15$ is true. In other words, that $p \geq 59$.  
 
There is a well-known expression for $k(G) = k(HV)$ which is a consequence of the so-called Clifford-Gallagher formula. Let $n(H,V)$ denote the number of $H$-orbits on $V$ and let $v_{1}, \ldots, v_{n(H,V)}$ be representatives of these orbits. Then \cite[Proposition 3.1b]{kgvbook} says that $k(HV) = \sum_{i=1}^{n(H,V)} k(C_{H}(v_{i}))$. This is at least $k(H) + n(H,V) - 1$. 

Theorem \ref{main} is then a consequence of the following result (with the roles of $H$ and $G$ interchanged). 

\begin{thm}
\label{main1}
Let $V$ be an irreducible and faithful $FG$-module for some finite group $G$ and finite field $F$ of characteristic $p$ at least $59$. Suppose that $p$ does not divide $|G|$. Then we have $k(G) + n(G,V) - 1 \geq 2 \sqrt{p-1}$ with equality if and only if $\sqrt{p-1}$ is an integer, $|V| = |F| = p$ and $|G| = \sqrt{p-1}$. 
\end{thm} 

Theorem \ref{main1} has implicitly been proved in \cite{HK} in case $G$ is solvable, without a consideration of when equality can occur. 

\section{Basic results, notations and assumptions}
\label{basic}

In the rest of the paper we are going to prove Theorem \ref{main1}. For this purpose let us fix some notations and assumptions. 

Let $V$ be an irreducible and faithful $FG$-module for some finite group $G$ and finite field $F$ of characteristic $p$. Suppose that $p$ does not divide $|G|$ and it is at least $59$. The size of the field $F$ will be denoted by $q$, the dimension of $V$ over $F$ by $n$, and the center of $GL(n,q)$ by $Z$. We denote the number of orbits of $G$ on $V$ by $n(G,V)$. We will use the following trivial observation throughout the paper. 

\begin{lem}
\label{trivial}
With the notations and assumptions above, $|V|/|G| \leq n(G,V)$.
\end{lem}

However we will also need a more sophisticated lower bound for $n(G,V)$. For this we must introduce some more notations (which will also be valid for the rest of the paper). 

Suppose that $G$ transitively permutes a set $\{ V_{1}, \ldots , V_{t} \}$ of subspaces of $V$ with $t$ an integer with $1 \leq t \leq n$ as large as possible with the property that $V = V_{1} \oplus \cdots \oplus V_{t}$. Let $B$ be the kernel of this action of $G$ on the set of subspaces. Note that $G/B$ is a transitive permutation group of degree $t$. The subgroup $B$ is isomorphic to a subdirect product of $t$ copies of a finite group $T$. In other words $B$ is isomorphic to a subgroup of $T_{1} \times \cdots \times T_{t}$ where for each $i$ with $1 \leq i \leq t$ the vector space $V_{i}$ is a primitive (faithful) $T_{i}$-module and $T_{i} \cong T$. Suppose that $T_{i}$ has $k$ orbits on $V_{i}$ (for each $i$). Let $H_{1}$ be the stabilizer of $V_{1}$ in $G$. Then the following is true.

\begin{lem}
\label{l99}
With the above notations and assumptions, $$n(H_{1},V_{1}) \leq \max \{ t+1, k  \} \leq \binom{t+k-1}{k-1} \leq n(G,V).$$
\end{lem}

\begin{proof}
For $k = n(H_{1},V_{1})$ this is \cite[Lemma 2.6]{F}. For a proof of this slightly stronger form we may assume that $G$ is as large as possible subject to the restrictions above (fixed $t$, $T$, $k$). Then $G \cong T \wr S_{t}$ and in this case $n(G,V)$ is precisely $\binom{t+k-1}{k-1}$.    
\end{proof}

When $G$ is solvable we will also use the following consequence of a result of S. M. Seager \cite[Theorem 1]{Seager}.

\begin{prop}   
\label{Seager}
Let $V$ be a faithful primitive $FG$-module for a finite solvable group $G$ not contained in $\Gamma L(1,p^{n})$ where $F$ is a field of prime order $p \geq 59$ and $|V| = p^{n}$. Then $p^{n/2}/12n < n(G,V)$. 
\end{prop}

As is suggested by Lemma \ref{trivial}, in various situations it will be useful to bound the size of $G$ from above. A useful tool in doing so is the following result of P. P. P\'alfy and Pyber \cite[Proposition 4]{PP}. 

\begin{prop}
\label{PP}
Let $X$ be any subgroup of the symmetric group $S_m$ whose order is coprime to a prime $p$. If $m > 1$ then $|X| < p^{m-1}$. 
\end{prop}

A third means to attack Theorem \ref{main1} is to bound $k(G)$. 

\begin{lem}
\label{abelian}
If $G$ has an abelian subgroup of index at most ${|V|}^{1/2}/(2 \sqrt{p-1})$ and $n(G,V) \leq 2\sqrt{p-1}$, then $2\sqrt{p-1} \leq k(G)$.
\end{lem}

\begin{proof}
If $G$ has an abelian subgroup $A$ with $|G:A| \leq {|V|}^{1/2}/(2 \sqrt{p-1})$, then $$|G|(2 \sqrt{p-1})/{|V|}^{1/2} \leq |A|.$$ Now $|A|/|G:A| \leq k(G)$, by a result of Ernest \cite[page 502]{Ernest} saying that whenever $Y$ is a subgroup of a finite group $X$ then we have $k(Y)/|X:Y| \leq k(X)$. This gives $(4 (p-1) |G|)/|V| \leq k(G)$. Then, by Lemma \ref{trivial}, we obtain $2\sqrt{p-1} \leq k(G)$. 
\end{proof}

\section{The class $\mathcal{C}_{q}$}

Our first aim in proving Theorem \ref{main1} is to describe (as much as possible) the possibilities for $G$ and $V$ with the condition that $n(G,V) < 2 \sqrt{q-1}$ where $q$ is the size of the underlying field $F$. For this we need to introduce a class of pairs $(G,V)$ which we denote by $\mathcal{C}_{q}$.  

In this paragraph we define a class of pairs $(G,V)$ where $V$ is an $FG$-module. Let $W$ be a not necessarily faithful but coprime $QH$-module for some finite field extension $Q$ of $F$ and some finite group $H$. We write $\mathrm{Stab}_{Q_{1}}^{Q}(H,W)$ for the class of pairs $(H_{1},W_{1})$ with the property that $W_{1}$ is a $Q_{1}H_{1}$-module with $F \leq Q_{1} \leq Q$ where $W_{1}$ is just $W$ viewed as a $Q_{1}$-vector space and $H_{1}$ is some group with the following property. If $\varphi : H_{1} \longrightarrow GL(W_{1})$ and $\psi : H \longrightarrow GL(W)$ denote the natural, not necessarily injective homomorphisms, then $\varphi(H_{1}) \cap GL(W) = \psi(H)$. We write $\mathrm{Ind}(H,W)$ for the class of pairs $(H_{1},W_{1})$ with the property that $W_{1} = \mathrm{Ind}_{H}^{H_{1}}(W)$ for some group $H_{1}$ with $H \leq H_{1}$. Finally, let $\mathcal{C}_{q}$ be the class of all pairs $(G,V)$ with the property that $V$ is a finite, faithful, coprime and irreducible $FG$-module so that $(G,V)$ can be obtained by repeated applications of $\mathrm{Stab}_{Q_{1}}^{Q_{2}}$ and $\mathrm{Ind}$ starting with $(H,W)$ where $W$ is a $1$-dimensional $QH$-module with $Q$ a field extension of $F$. 

If $(G,V) \in \mathcal{C}_{q}$ then there exist a sequence of field extensions $$F_{q_{m}} \geq F_{q_{m-1}} \geq \ldots \geq F_{q_{0}} =F,$$ a normal series $1 < N_{0} \triangleleft N_{1} \triangleleft \ldots \triangleleft N_{2m-1} = G$, and integers $n_{1}, \ldots, n_{m}, n_{m+1} = 1$ so that the following hold. The normal subgroup $N_{0}$ of $G$ is a subgroup of the direct product of $\log |V| / \log q_{m}$ copies of a cyclic group of order $q_{m}-1$. For each $i$ with $1 \leq i \leq m$ the factor group $N_{2i-1}/N_{2i}$ is a subgroup of the direct product of 
$n_{i} \leq \log |V|/\log q_{m-i+1}$ copies of a cyclic group of order $\log q_{m-i+1}/ \log q_{m-i}$ and the factor group $N_{2i}/N_{2i-1}$ is a subgroup of a permutation group on $n_{i}$ points which is a direct power of $n_{i+1}$ copies of a permutation group on $n_{i}/n_{i+1}$ points.

The main results of this section are Lemmas \ref{order} and \ref{new}. 

\begin{lem}
\label{order}
Let $(G,V) \in \mathcal{C}_{q}$ and $n(G,V) < 2\sqrt{q-1}$. If $q \geq 59$, then $|G| < {|V|}^{3/2}$. 
\end{lem}   

\begin{proof}
Fix an $F_{q_{0}}$-vector space $V$ of dimension $n$ where $q_{0} = q$. Suppose that $(G,V) \in \mathcal{C}_{q}$ with $n(G,V) < 2\sqrt{q-1}$ and $G$ of maximal possible size. Then there exists a sequence of field extensions $F_{q_{m}} \geq F_{q_{m-1}} \geq \ldots \geq F_{q_{0}}$ so that $$|G| \leq {(q_{m}-1)}^{\log |V| / \log q_{m}} \cdot \Big( \prod_{i=1}^{m} {(\log q_{i} / \log q_{i-1})}^{\log |V| / \log q_{i}} \Big) \cdot p^{\log |V| / \log q_{m} - 1}$$ where the first factor is equal to the size of the direct product of $\log |V| / \log q_{m}$ copies of a cyclic group of order $q_{m}-1$, the second factor is an upper bound for the product of all the factors with which the sizes of the relevant groups increase by taking normalizers when viewing the linear groups over smaller fields, and the third factor is the product of the sizes of all factor groups (viewed as permutation groups) which arise after inducing smaller modules (this product is at most the size of a $p'$-subgroup of the symmetric group on $\log |V| / \log q_{m}$ points which we can bound using Proposition \ref{PP}).  

We now proceed to bound the three factors in the product above. 
The first factor is clearly less than $|V|$. Let us consider the second factor. Define the positive integers $k_{1}, \ldots , k_{m}, k_{m+1}$ so that $q_{1}= q^{k_{1}}$, $q_{2} = q^{k_{1}k_{2}}, \ldots , q_{m} = q^{k_{1}k_{2} \cdots k_{m}}$, and $|V| = q^{k_{1}k_{2} \cdots k_{m} k_{m+1}}$. We may assume that all the $k_{i}$'s are at least $2$ for $1 \leq i \leq m$ (while we allow $k_{m+1}$ to be $1$). Then we can write the second factor as $$\prod_{i=1}^{m} {k_{i}}^{k_{i+1} \cdots k_{m+1}} \leq \prod_{i=1}^{m} {k_{i}}^{n/(k_{1} \cdots k_{i})}$$ where $n = \log |V| / \log q$. But by taking logarithms it is easy to see that $$\prod_{i=1}^{\infty} {n_{i}}^{1/(n_{1} \cdots n_{i})} \leq 3^{2/3}$$ for any sequence $n_{1}, n_{2}, \ldots$ of integers at least $2$. Thus the second factor is at most $3^{2n/3} < |V|^{0.18}$ since $q \geq 59$. 

Suppose first that $q_{m} \geq q^{4}$. Then we can show that $|G| < {|V|}^{1.39}$. This is clear for $q_{m} \geq q^{10}$ since the second factor considered above is less than $|V|^{0.18}$ while the third factor is less than ${|V|}^{1/10}$. By bounding the second factor more carefully in cases $q_{m} = q^{i}$ ($4 \leq i \leq 9$), we see that it is less than $|V|^{0.39-1/i}$.

Thus we may assume that $q_{m} = q^{3}, q^{2}$ or $q$. In the first two cases $m = 1$ while in the third, $m = 0$. 

Suppose that the first case holds. Then we can bound the second factor by $3^{n/3} < |V|^{0.09}$. By Lemma \ref{l99} and by using the fact that $n(G,V) < 2\sqrt{q-1}$, we certainly have $n/3 < \ell:= 2 \sqrt{q-1}$. So the third factor is at most $$(n/3)! < {\ell}^{n/3} < {(4 \cdot q)}^{n/6} < {|V|}^{1/4}$$ since $q \geq 59$. So we get $|G| < |V|^{1.34}$.

Suppose that the second case holds. Then we can bound the second factor by $2^{n/2} < |V|^{0.09}$. By Lemma \ref{l99} and by using the fact that $n(G,V) < 2\sqrt{q-1}$, we certainly have $n/2 < \ell:= 2 \sqrt{q-1}$. So the third factor is at most $$(n/2)! < {\ell}^{n/2} < {(4 \cdot q)}^{n/4} < {|V|}^{0.34}.$$ So we get $|G| < |V|^{1.43}$. 

Suppose that the third case holds. Then the second factor is $1$. Also, by Lemma \ref{l99}, we can replace the third factor by $n!$ where $n < 2 \sqrt{q-1}$. Here $2 \sqrt{q-1} \geq 15$. This gives $n! < {(\sqrt{q-1})}^{n} < q^{n/2} = {|V|}^{1/2}$. We get $|G| < {|V|}^{3/2}$.
\end{proof}

The following can be considered as a refined version of Lemma \ref{order}.

\begin{lem}
\label{new}
Let $(G,V) \in \mathcal{C}_{q}$ and $n(G,V) < 2\sqrt{q-1}$. If $p \geq 59$, then at least one of the following holds.
\begin{enumerate}
\item $G$ has an abelian subgroup of index at most ${|V|}^{1/2}/(2\sqrt{p-1})$.

\item $|F| = p$, the module $V$ is induced from a $1$-dimensional module, and $G$ has a factor group isomorphic to $A_{n}$ or $S_{n}$ where $n = \dim_{F} (V)$. In this case we either have $n=1$, or $15 \leq n \leq 180$ and $p < 8192$.
\end{enumerate}
\end{lem}

\begin{proof}
If $G \leq \Gamma L(1,q^{n})$, then the result is clear, since $n \leq {|V|}^{1/2}/(2\sqrt{p-1})$ ((1) is satisfied).
  
Let us consider the proof (and the notations) of Lemma \ref{order}. Clearly, an upper bound for the index of an abelian (subnormal) subgroup of $G$ is the product of the second and third factors. For $q_{m} \geq q^{4}$ this was ${|V|}^{0.39}$, for $q_{m} = q^{3}$ this was ${|V|}^{0.34}$, and for $q_{m} = q^{2}$ this was ${|V|}^{0.43}$. These are at most ${|V|}^{1/2}/(2\sqrt{p-1})$ unless $n \leq 6$ (in the first case), $n=3$ (in the second case), and $n \leq 8$ (in the third case). In all these exceptional cases we have $G \leq \Gamma L(1,q^{n})$ (the case treated in the previous paragraph) unless $q_{m} = q^{2}$ and $n = 4$, $6$, or $8$. But in all these exceptional cases there exists an abelian (subnormal) subgroup of index at most $2^{n} (n/2)! < q^{n/2}/(2 \sqrt{p-1})$ where this latter inequality follows from $q \geq p \geq 59$ and $n \leq 8$. Thus (1) is satisfied in all these cases, and we may assume that $q_{m} = q$ in case $(G,V) \in \mathcal{C}_{q}$.  


Now let $t$ and $B$ be defined for $G$ as in Section \ref{basic}. By Lemma \ref{l99}, we may assume that $t < 2 \sqrt{p-1} - 1$. Put $\ell$ to be the integer part of $2 \sqrt{p-1} - 1$. Then it is easy to see that $|G/B| \leq {\ell!}^{t/\ell} < {(\ell/2.2549)}^{t}$ since $p \geq 59$. This gives $|G/B| < {0.89}^{t} \cdot p^{t/2}$. 


Suppose that $t=n$. Then $G$ contains an abelian (normal) subgroup of index less than ${0.89}^{n} \cdot {p}^{n/2} \leq {q}^{n/2}/(2\sqrt{p-1})$ unless ${1.27}^{n} < 4 (p-1)$ (in which case this previous inequality fails). By taking logarithms of both sides we get $n < 10 \log p$. But then $|G/B| < {((10/2.2549)\log p)}^{n} < {(4.5 \log p)}^{n}$.  

Suppose for a contradiction that part (1) fails. Then $$q^{n/2}/(2 \sqrt{p-1}) < |G/B| < {(4.5 \log p)}^{n}.$$ This gives ${(\sqrt{q}/(4.5 \log p))}^{n} < 2 \sqrt{p-1}$. But on the other hand we also have $|G/B| \leq n!$ which, together with our assumption, gives the inequality $p^{n-1} < 4 (n!)^{2}$. Since $p \geq 59$, we certainly have $59^{n-1} < 4 {(n!)}^{2}$. From this we get $n \geq 15$.
Then ${(\sqrt{q}/(4.5 \log p))}^{15} < 2 \sqrt{p-1}$, which forces $q = p < 8192$ and thus $n = t \leq 180$.  

It is easy to see that a transitive subgroup of $S_n$ not containing $A_{n}$ has index at least $3n$ for $n \geq 15$. (This is clear for a primitive subgroup by the bound of Praeger and Saxl \cite{PS}, while for imprimitive groups a more direct calculation is necessary.) So if $G/B$ does not contain the alternating group $A_n$, then we can refine our upper bound above for $|G/B|$ by multiplying the result by $1/3n$. But then $9 \cdot n^{2} \cdot {1.27}^{n} < 4 (p-1)$ follows. However since $n \geq 15$ we also get $73026 <  n^{2} \cdot {1.27}^{n} < 4 (p-1)$ which forces $18257 \leq p$. But this is a contradiction because we already deduced that $p < 8192$. This proves the result in case $t=n$. 
\end{proof}

\section{Some absolutely irreducible representations}

As mentioned earlier we will be interested in pairs $(G,V)$ for which $n(G,V) < 2 \sqrt{q-1}$. In this section we consider two special cases, the case when $G$ is a central product of an almost quasisimple group $H$ and $Z$ and the case when $G$ is the normalizer of a group of symplectic type. We will also make some more assumptions on the $FG$-module $V$.  

\begin{prop}
\label{quasisimple}
Suppose that $p$ is a prime at least $59$. Let $H$ be a finite subgroup of $GL(n,q)$ with generalized Fitting subgroup a quasisimple group where $q$ is a power of $p$. Put $G = Z \circ H$ where $Z$ is the multiplicative group of $F$. Furthermore suppose that $V$ is an absolutely irreducible $FT$-module for every non-central normal subgroup $T$ of $G$. Suppose also that $|G|$ is not divisible by $p$. Then $n(G,V) \geq 2\sqrt{q-1}$ unless possibly if $n=2$, $q$ is in the range $59 \leq q \leq 14 389$, it is congruent to $\pm 1$ modulo $10$, and $G = Z \circ 2.A_{5}$. 
\end{prop}

\begin{proof}
First suppose that $H$ cannot be realized over a proper subfield of $F$. 

In this paragraph suppose also that $(n,H)$ is different from $(2,2.A_{5})$, $(3,3.A_6)$, $(3,L_{2}(7))$, or $(4,2.S_{4}(3))$. Let $P(V)$ denote the set of $1$-dimensional subspaces of $V$. Since $|H|$ is not divisible by $p$ and $p \geq 59$, we see by \cite[Satz 3.4]{M}, that all orbits of $G$ on $P(V)$ have lengths less than $(q^{n-1}-1)/(q-1)$. Thus the number of orbits of $G$ on $P(V)$ is larger than $(q^{n}-1)/(q^{n-1}-1)$. But then $n(G,V)$ is larger than $(q^{n}-1)/(q^{n-1}-1) \geq 2 \sqrt{q-1}$.

Suppose that $n = 2$ and $G = Z \circ 2.A_{5}$. We may assume that $n(G,V) < 2\sqrt{q-1}$. From this we get $|V|/|G| < 2 \sqrt{q-1}$. It readily follows that $q < 14400$. Since $q$ must be a prime power, we must have $59 \leq q \leq 14 389$. According to Dickson's theorem \cite[Kapitel II, 8.27]{Huppert} $q$ must be congruent to $\pm 1$ modulo $10$. This accounts for the exception in the statement of the proposition.

Suppose that $n=3$ and $G = Z \circ 3.A_{6}$. From the inequality $|V|/|G| < 2 \sqrt{q-1}$ it follows that $q \leq 80$. Since $q$ is a prime power, we must have $59 \leq q \leq 79$. However only $p=61$ is to be considered since $\sqrt{-3}$ and also $\sqrt{5}$ must lie in $F$. In this case a direct computation shows that there are $21 > 2 \sqrt{60}$ orbits of $G$ on $P(V)$.

Suppose that $n=3$ and $G = Z \times L_{2}(7)$. From the inequality $|V|/|G| < 2 \sqrt{q-1}$ it follows that $q \leq 48$. A contradiction. 

Suppose that $n=4$ and $G = Z \circ 2.S_{4}(3)$. From the inequality $|V|/|G| < 2 \sqrt{q-1}$ it follows that $q \leq 76$. Since $q$ is a prime power, we must have $59 \leq q \leq 73$. However only the cases $p=61$, $67$, and $73$ are to be considered since
$\sqrt{-3}$ must lie in $F$. In these cases there are $30$, $33$, and $43$ orbits of $G$ on $P(V)$, respectively. These are all greater than $2 \sqrt{72}$. 

Now suppose that $H$ can be realized over a proper subfield of $F$. Then clearly $q \geq 59^{2}$. Let $S$ be the generalized Fitting subgroup of $H$ which by assumption is quasisimple. We now discuss the possibilities for $S$ according to the classification.

If $n=2$ then, by Dickson's theorem \cite[Kapitel II, 8.27]{Huppert}, $S$ is a covering group of $A_5$ and $H = S$. This is an exception in the statement of the proposition since as before we get $q \leq 14389$ and $q \equiv \pm 1 \pmod{10}$. From now on assume that $n \geq 3$. 

Since $q \geq 59^{2}$, it can easily be checked, just by order considerations and using the fact that $|G|$ is coprime to $p$, that none of the (generic examples of) groups $G$ with $S$ appearing in Table 2 of \cite{HM1} have fewer than $2 \sqrt{q-1}$ orbits on $V$. Then, using Table 3 of \cite{HM2} together with the condition that $q \geq 59^{2}$, one can check, essentially just by comparing $\log_{10}(q^{n-2})$ and $\log_{10}(|G|)$, that no group $G$ has fewer than $2 \sqrt{q-1}$ orbits on $V$ with $n \leq 250$.  

So assume that $n > 250$. We can rule out $S$ being a covering group of a sporadic simple group since $|G|$ is much smaller than $59^{498}$. For a similar reason as when considering Table 2 of \cite{HM1}, we see that $S$ cannot be a covering group of an alternating group $A_m$ (for we can assume that $m \geq 9$ and so $n \geq m-2$ by \cite[Proposition 5.3.7 (i)]{KL}). 

Suppose that $S$ is a covering group of a classical group $Cl(d,r)$ where $r$ is a prime power and $d$ is chosen as small as possible (here $d$ is the dimension of the vector space naturally associated to the classical group). If $d \geq 6$ then $$n \geq \max \{251, (r^{d/2}-1)/2 \}$$ by \cite[Corollary 5.3.10 (iv)]{KL} and by $n > 250$. But then $q^{n-3} > r^{d^{2}}$ follows by using the fact that $q \geq 59^{2}$. This implies that we must have $d \leq 5$.

So suppose that $d \leq 5$. Then \cite[Table 5.3.A]{KL} shows that $n \geq \max \{ 251, (r-1)/2 \}$. But then $q^{n-3} > r^{d^{2}}$ certainly follows for $r \geq 32$. So suppose that $r < 32$. Then $q^{n-3} \geq 59^{496} > 32^{25} \geq r^{d^{2}}$. This finishes the treatment of the case when $S$ is a covering group of a classical group. 

Suppose that $S$ is a covering group of an exceptional simple group of Lie type. Then \cite[Tables 5.1.B and 5.3.A]{KL} can be used to show that $G$ must have at least $2 \sqrt{q-1}$ orbits on $V$.
\end{proof}

Let us now turn to our second important case of an absolutely irreducible $FG$-module $V$. Suppose that the group $G$ has a unique normal subgroup $R$ which is minimal subject to being non-central. Suppose that $R$ is an $r$-group of symplectic type for some prime $r$ (this is an $r$-group all of whose characteristic abelian subgroups are cyclic). Suppose that $V$ is an absolutely irreducible $FR$-module. Let $|R/Z(R)| = r^{2a}$ for some positive integer $a$. Then the dimension of the module is $n = r^{a}$. Suppose that $Z \leq G$. The group $G/(R Z)$ can be considered as a subgroup of the symplectic group $\spl{2a}{r}$. As always, we assume that $q \geq p \geq 59$.

\begin{prop}
\label{symplectic}
Suppose that $V$ and $G$ satisfy the assumptions of the previous paragraph. If $n(G,V) < 2\sqrt{q-1}$, then $n=2$, $59 \leq q=p \leq 2 297$, and $|G/Z| \leq 24$.  
\end{prop}

\begin{proof}
Suppose that $V$ and $G$ satisfy the assumptions of the paragraph preceding the statement of the proposition. Then $|V|/|G| < 2 \sqrt{q-1}$. 

Suppose first that $(r,a)$ is different from any of the pairs $(2,1)$, $(3,1)$, and $(2,2)$. Then $|G| < q \cdot r^{2a} \cdot |\spl{2a}{r}| < q \cdot r^{2a^{2} + 3a}$. We wish to show that this is less than $q^{r^{a} -1} \leq |V|/(2 \sqrt{q-1})$. By taking logarithms of both sides, it is sufficient to see the inequality $(2a^{2} + 3a) \log r < (r^{a}-2) \log q$. But this is true by using the assumption that $q \geq p \geq 59$. This is a contradiction to the fact that $|V|/|G| < 2 \sqrt{q-1}$. 

If $(r,a) = (2,2)$ then a more careful but similar computation as in the previous paragraph yields a contradiction. For $(r,a) = (3,1)$ we do the same and get a contradiction whenever $q \geq 61$. Also, $q$ cannot be $59$ in this case since $3$ does not divide $58$. 

So only $(r,a) = (2,1)$ can occur. In this case we must have $n=2$, $|G/Z| \leq 24$, and thus $q$ is in the range $59 \leq q=p \leq 2297$. 
\end{proof}

\section{Bounding $n(G,V)$}

The purpose of this section is to describe as much as possible pairs $(G,V)$ for which $n(G,V) < 2 \sqrt{q-1}$.

\begin{thm}
\label{orbit}
Let $V$ be a finite, faithful, coprime and irreducible $FG$-module. Suppose that the characteristic $p$ of the underlying field $F$ is at least $59$. Put $q = |F|$ and $|V| = q^{n}$. Let the center of $GL(n,q)$ be $Z$. Then $n(G,V) \geq 2\sqrt{q-1}$ unless possibly if one of the following cases holds. 
\begin{enumerate}

\item $(G,V) \in \mathcal{C}_{q}$;

\item $V = \mathrm{Ind}_{H}^{G}(W)$ for some $2$-dimensional $FH$-module $W$ where $H$ is as $G$ in Proposition \ref{quasisimple} or Proposition \ref{symplectic} satisfying one of the following.

\begin{enumerate}
\item $59 \leq q \leq 14 389$, $q \equiv \pm 1 \pmod{10}$, and $2.A_{5} \leq H/C_{H}(W) \leq Z \circ 2.A_{5}$; 

\item $59 \leq q=p \leq 2 297$ and $|(H/C_{H}(W))/Z(H/C_{H}(W))| \leq 24$. 
\end{enumerate}
\end{enumerate}
\end{thm}

In order to prove Theorem \ref{orbit} we need a bound on the orders of groups among the exceptions in the statement of the theorem. The following extends Lemma \ref{order}.

\begin{lem}
\label{seged}
Let $(G,V)$ be a pair among the exceptions in Theorem \ref{orbit}, satisfying $n(G,V) < 2\sqrt{q-1}$. Then $|G| < |V|^{3/2}$.
\end{lem}

\begin{proof}
If $(G,V)$ is of type (1) then Lemma \ref{order} gives the result. If $(G,V)$ is of type (2/a) or (2/b) then it is easy to see that $|G| < |V|^{3/2}$ by using Proposition \ref{PP} and the fact that $p \geq 59$.
\end{proof}

We can now turn to the proof of Theorem \ref{orbit}. In this we follow the reduction argument found in \cite[Section 6]{HP}.

Let $G$ be a counterexample to Theorem \ref{orbit} with $n$ minimal. 

Suppose that $V$ is an imprimitive $FG$-module which is induced from a primitive $FH$-module $W$ for some proper subgroup $H$ of $G$. If $n(H,W) \geq 2\sqrt{q-1}$ then $n(G,V) \geq 2\sqrt{q-1}$, by Lemma \ref{l99}. So assume that $n(H,W) < 2 \sqrt{q-1}$. By the minimality of $n$, the pair $(H/C_{H}(W),W)$ must be of type (1) or (2) of the statement of the theorem. But then $(G,V)$ is also of type (1) or (2). A contradiction. 

So we may assume that $V$ is a primitive $FG$-module.  

We first claim that we can assume that every irreducible
$FN$-submodule of $V$ is absolutely irreducible for any normal
subgroup $N$ of $G$. For this purpose let $N$ be a normal subgroup
of $G$. Then $V$ is a homogeneous $FN$-module, so $V = V_1 \oplus
\cdots \oplus V_m$, where the $V_i$'s are isomorphic irreducible
$FN$-modules. Let $K \simeq \mathrm{End}_{FN}(V_1)$. Assuming that
the $V_i$'s are not absolutely irreducible, $K$ is a proper field
extension of $F$, and $C_{GL(V)}(N)=\mathrm{End}_{FN}(V)\cap GL(V)
\simeq GL(r,K)$ for some $r$. Furthermore, $L = Z(C_{GL(V)}(N)) \simeq
Z(GL(r,K)) \simeq K^{\times}$. Now, by using $L$, we can extend
$V$ to a $K$-vector space of dimension $\ell:=\dim_K V < n$. As
$G\leq N_{GL(V)}(L)$, in this way we get an inclusion $G \leq
\Gamma L(\ell,K)$. Now $G$ contains the normal subgroup $H = G \cap GL(\ell,K)$ 
of index at most $n$. Clearly $V$
is a homogeneous and faithful $KH$-module. Let $W$ be a simple $KH$-submodule of $V$. Then, by the minimality of $n$, 
we get $n(H,V) \geq n(H,W) \geq 2 \sqrt{|K|-1}$ unless $(H,W)$ is one of the examples listed in the statement of the theorem. If $H$ is none of the possibilities listed in the statement of the theorem, then $n(G,V) \geq n(H,V)/n \geq 2 \sqrt{q-1}$, a contradiction, since we are assuming $p \geq 59$. If $(H,W)$ is of possibility (1) then so is $(G,V)$ of possibility (1) unless $W < V$. If $W < V$ and $W$ is not of dimension $1$ over $K$ then Lemma \ref{order} shows that $n(H,V) \geq |V|/|H| \geq 2 \sqrt{|K|-1}$, and so $n(G,V) \geq n(H,V)/n \geq 2 \sqrt{q-1}$, as before. If $W$ is of dimension $1$ over $K$ then a more careful consideration is necessary to obtain the same conclusion. If $H$ is of possibility (2) of the statement of the theorem, then $H$ is of index $2$ in $G$ and so $|G| \leq 120(q^{2}-1)$. But then $n(G,V) > |V|/|G| > q^{2}/120 \geq 2 \sqrt{q-1}$ since $q \geq 59$. This shows the claim. 

Let $N$ be a normal subgroup of $G$ and let $V = V_1\oplus \cdots
\oplus V_r$ be a direct sum decomposition of $V$ into isomorphic
absolutely irreducible $FN$-modules. By choosing a suitable basis
in $V_1,V_2,\ldots, V_r$, we can assume that $G \leq GL(n,F)$
such that any element of $N$ is of the form $A\otimes I_r$ for
some $A\in N_{V_1}\leq GL(n/r,F)$. By using \cite[Lemma
4.4.3(ii)]{KL} we get
\[N_{GL(n,F)}(N)=\{B\otimes C\,|\,B\in N_{GL(n/r,F)}(N_{V_1}),\ C\in GL(r,F) \}.\] Let
\[
G_1=\{g_1\in GL(n/r,F)\,|\,\exists g\in G,g_2\in GL(r,F)
\textrm{ such that }g=g_1\otimes g_2\}.
\]
We define $G_2\leq GL(r,F)$ in an analogous way. Then $G\leq
G_1\otimes G_2$. Here $G_1$ and $G_2$ are not homomorphic images
of $G$, since $g=g_1\otimes g_2=\lambda g_1\otimes\lambda^{-1}g_2$
for any $\lambda\in F^{\times}$, so the map $g=g_1 \otimes g_2
\mapsto g_1$ is not well-defined. However, they both have orders coprime to $p$. 
Since $G_{1} \otimes G_{2}$ preserves a tensor product structure $V = W_{1} \otimes W_{2}$, so does $G$. 

We claim that $G$ does not preserve a proper tensor product structure. For a proof suppose that $G$ preserves a tensor product structure $V = W_{1} \otimes W_{2}$ with $r_{1} = \dim W_{1} > 1$ and $r_{2} = \dim W_{2} > 1$. Without loss of generality assume that $r_{1} \leq r_{2}$ and $n = r_{1} r_{2}$. Then $G \leq G_{1} \otimes G_{2}$ for some groups $G_{1}$ and $G_{2}$ acting on $W_{1}$ and $W_{2}$ respectively. Assume also that these groups have orders coprime to $p$. We also assume that $G$ acts primitively and irreducibly on $V$ and $Z \leq G$. 
Notice that the $G_{i}$'s act irreducibly on the $W_{i}$'s (for if $0< U_{1} < W_{1}$ would be a $G_{1}$-submodule then $U_{1} \otimes W_{2}$ would be a $G_{1} \otimes G_{2}$-submodule). Also, the $G_{i}$'s act primitively on the $V_{i}$'s. (For if $G_{1}$ would act imprimitively on $W_{1}$, say, then there would be a proper subspace $U_{1}$ in $W_{1}$ whose stabilizer has index $|W_{1}|/|U_{1}|$. But then $U_{1} \otimes W_{2}$ would be a subspace of $V$ whose stabilizer in $G_{1} \otimes G_{2}$ has the same index. But then, by \cite[Theorem 3, page 105]{Suprunenko}, we see that $G_{1} \otimes G_{2}$, and in particular $G$, acts imprimitively on $V$, a contradiction.) If $n(G_{i},W_{i}) \geq 2 \sqrt{q-1}$ for any of the $i$'s, then we are done. (For if $n(G_{1},W_{1}) \geq 2 \sqrt{q-1}$, say, and $v_{1}, \ldots , v_{f}$ are representatives of $f$ orbits of $G_{1}$ on $W_{1}$ with $f \geq 2 \sqrt{q-1}$, then $v_{1} \otimes w, \ldots, v_{f} \otimes w$ will be representatives of $f$ orbits of $G$ on $V$ where $w$ is a non-zero vector in $W_{2}$.) So by the minimality of $n$ we know that both $G_{1}$ and $G_{2}$ are exceptions in the statement of the theorem. If $G$ is solvable, then Proposition \ref{Seager} gives a contradiction (since $n \geq 4$). So we may assume that $G$ is non-solvable. Notice that $|G| \leq (|G_{1}| \cdot |G_{2}|)/(q-1)$. 

Let $r_{1} = 2$. Then $|G| < 2 \cdot q^{(3/2) r_{2} + 1}$ by Lemma \ref{seged}. But then if $r_{2} \geq 5$ then $|V|/|G| > 2 \sqrt{q-1}$, a contradiction. We get the same conclusion when $r_{1} \geq 3$ and $r_{2} \geq 5$ (apply Lemma \ref{seged}). So we conclude that $2 \leq r_{1} \leq r_{2} \leq 4$. In fact, since $G$ is non-solvable, this forces $r_{1} = 2$ and $G_{1}$ of type (2/a).  

Let $r_{2} = 2$. To maximize $|G|$ we may assume that $G_{2}$ is of type (1) or (2/a). But then we again have $2\sqrt{q-1} < |V|/|G|$, a contradiction. So $r_{2} = 3$ or $4$. But then $G_{2}$ is of type (1). In both cases $G_{2}$ must be solvable. Let $r_{2} = 4$. If $G_{2}$ is not a semilinear group of order dividing $4(q^{4}-1)$, then Proposition \ref{Seager} gives what we want. Otherwise $2 \sqrt{q-1} \leq q^{8}/|G|$. So let $r_{2} = 3$. Then $|G_{2}| \leq 3(q^{3}-1)$ (since $G_{2}$ is primitive) and so $2 \sqrt{q-1} \leq q^{6}/|G|$.

We conclude that $G$ does not preserve a proper tensor product structure.

From now on assume that $N$ is a normal subgroup of $G$ which is
minimal with respect to being non-central. Then $N/Z(N)$ is a
direct product of isomorphic simple groups. 

If $N$ is abelian then it is central in $G$. A contradiction. 

If $N/Z(N)$ is elementary abelian of rank at least $2$, then $G$
is of symplectic type and Proposition \ref{symplectic} gives us a contradiction. 

Now let $N/Z(N)$ be a direct product of $m\geq 2$ isomorphic
non-abelian simple groups. Then $N=L_1\star L_2\star \cdots \star
L_m$ is a central product of isomorphic groups such that for every
$1\leq i\leq m$ we have $Z\leq L_i,\ L_i/Z$ is simple.
Furthermore, conjugation by elements of $G$ permutes the subgroups
$L_1,L_2,\ldots, L_m$ in a transitive way. By choosing an
irreducible $FL_1$-module $V_1\leq V$, and a set of coset
representatives $g_1=1,g_2,\ldots,g_m\in G$ of $G_1=N_G(V_1)$ such
that $L_i=g_iL_1g_i^{-1}$, we get that $V_i:=g_iV_1$ is an
absolutely irreducible $FL_i$-module for each $1\leq i\leq m$.
Now, $V\simeq V_1\otimes V_2\otimes \cdots \otimes V_m$ and $G$
permutes the factors of this tensor product. It follows that $G$
is embedded into the central wreath product $G_1\wr_c S_m$ and that $G_{1}$ is non-solvable. Now $G_{1}$ acts irreducibly on $V_{1}$ for otherwise there are proper $G_{1}$-submodules $W_{i}$ of $V_{i}$ for each $i$ with $1 \leq i \leq m$, so that $W = W_{1} \otimes \cdots \otimes W_{m}$ is a proper $G$-submodule of $V$. If $n(G_{1},V_{1}) \geq 2 \sqrt{q-1}$, then so is $n(G,V) \geq 2 \sqrt{q-1}$. (For let $v_{1}, \ldots, v_{r}$ be members of $n(G_{1},V_{1})-2$ non-trivial orbits of $G_{1}$ on $V_{1}$. Clearly $r > 2$ and these vectors are non-zero and pairwise not multiples of each other. Let $v_{r+1}$ be a fixed non-zero vector not in an orbit of any of the vectors listed above. Then the vectors $w_{i} = v_{i} \otimes v_{r+1} \otimes \ldots \otimes v_{r+1}$ for $1 \leq i \leq r+1$ are all non-zero and are all in different $G$-orbits.) So $G_{1}$ is a group among the exceptions in Theorem \ref{orbit}. So we have $|G_{1}| \leq {|V_{1}|}^{3/2}$ by Lemma \ref{seged}. Using this we can show that $|V|/|G| > 2 \sqrt{q-1}$ provided that $\dim V_{1} \geq 4$. So assume that $\dim V_{1} = 2$ or $3$. Since $G_{1}$ is non-solvable,   
we must then have $\dim V_{1} = 2$ and $G_{1} = Z \circ 2.A_{5}$. But then $|G| \leq (q-1) 60^{m} p^{m-1}$ where the last factor follows from Proposition \ref{PP} and $n = 2^{m}$. 
However when $m=2$ then by using just a factor of $2$ in place of $p^{m-1}$, we get $2 \sqrt{q-1} \leq |V|/|G|$. We get the same conclusion in case $m \geq 3$, by Proposition \ref{PP}. 

The remaining case is when $N/Z(N)$ is a non-abelian finite simple group. But then the generalized Fitting subgroup of $G$ is a central product of the center of $G$ with a quasisimple group (by the above reductions) and Proposition \ref{quasisimple} yields a contradiction.  

This proves Theorem \ref{orbit}.

\section{Bounding $k(G)$}

In order to prove Theorem \ref{main1} we now also have to take $k(G)$ into account. 

\begin{thm}
\label{main3}
Let $V$ be an irreducible and faithful $FG$-module for some finite group $G$ and finite field $F$ of characteristic $p$ at least $59$. Suppose that $p$ does not divide $|G|$. Then we have at least one of the following. 
\begin{enumerate}
\item $n(G,V) \geq 2 \sqrt{p-1}$.

\item $k(G) \geq 2 \sqrt{p-1}$.

\item $\sqrt{p-1}$ is an integer, $|V| = |F| = p$ and $|G| = \sqrt{p-1}$.

\item Case (2/a) of Theorem \ref{orbit} holds with $p=59$ and $1 < t \leq 14$, or $p=61$ and $t=1$, or $61 < p \leq 119$ and $t \leq 4$.

\item Case (2/b) of Theorem \ref{orbit} holds with $t \leq 4$. 
\end{enumerate}
\end{thm}  

\begin{proof}
Let $V$ be an irreducible and faithful $FG$-module as in the statement of the theorem. Suppose that $n(G,V) < 2 \sqrt{p-1}$. Suppose also that case (3) is not satisfied. More in general, suppose that $|V| = |F| = p$ is not satisfied. 

We are then in one of the two exceptional cases of Theorem \ref{orbit}. First suppose that $(G,V) \in \mathcal{C}_{q}$. Then case (1) or case (2) of Lemma \ref{new} holds. In case (1) we may apply Lemma \ref{abelian}. So suppose that case (2) of Lemma \ref{new} holds.

Suppose that $G/B$ contains $A_{n}$. Then $$k(G) \geq k(G/B) \geq k(S_{n})/2 \geq 385/2 > 2 \sqrt{p-1},$$ for $n \geq 18$ and $p < 8192$. So we must have $n = 15$, $16$, or $17$.

Since $k(G) \geq k(G/B) \geq k(A_{15}) = 94$, we may assume that $94 < 2 \sqrt{p-1}$, that is, $2210 < p$. We may assume that $p^{n-1} < 4 {(n!)}^{2}$ (otherwise we are in case (1) of Lemma \ref{new}). By the fact that $2210 < p$, we get $2210^{n-1} < 4 {(n!)}^{2}$. But this is a contradiction for $n = 15$, $16$, or $17$.

We are now in case (2) of Theorem \ref{orbit}. 

First we consider case (2/a) of Theorem \ref{orbit}. 

Let us first assume that $V$ is a primitive $FG$-module. Let $C$ be the center of $G$. Then $G$ contains at least $(|C|/2) \cdot k(A_{5}) = (5/2)|C|$ conjugacy classes. Thus we may assume that $|C| < (4/5) \sqrt{p-1}$. But we also have $|V|/(2 \sqrt{p-1}|G/C|) < |C|$. From this we have $|V| < (8/5) (p-1) \cdot 60$, that is $q^{2} < 96 (p-1)$. Thus we certainly have $p \leq 96$ but also $q=p$. Thus we are left with the cases $q=p= 59$, $71$, $79$, and $89$ (note that we are excluding $61$ here).

Let $q=59$. Then $|C| \leq 6$ by the previous paragraph. But since $|C|$ must divide $q-1 = 58$ and is even, we have $|C|=2$. So $G$ has at least (if not exactly) $29$ non-trivial orbits on $V$, which is larger than $2\sqrt{58}$.

Let $q=71$. Then $|C| \leq 6$. But since $|C|$ must divide $q-1 = 70$ and is even, we have $|C|=2$.  So $G$ has at least $42$ non-trivial orbits on $V$, which is larger than $2\sqrt{70}$. 

Let $q=79$. Then $|C| \leq 7$. But since $|C|$ must divide $q-1 = 78$, we have $|C| \leq 6$. So $G$ has at least $18$ non-trivial orbits on $V$, which is larger than $2\sqrt{78}$.

Let $q=89$. Then $|C| \leq 7$. But since $|C|$ must divide $q-1 = 88$, we have $|C| \leq 4$. So $G$ has at least $33$ non-trivial orbits on $V$, which is larger than $2\sqrt{88}$.

Now assume that $V$ is an imprimitive $FG$-module. Let $T$, $t$, $n$, $B$ and $k$ be as above. So $n \geq 4$ and $t \geq 2$.

Suppose that $p > 1000$. Then the number of orbits of $T$ on $V_{1}$ is at least $3$ (since $T$ cannot be a transitive linear group by Hering's theorem (see \cite[Chapter XII]{HB})). But then $n(G,V) \geq \binom{t+2}{2}$ by Lemma \ref{l99}. So we may assume that $2\sqrt{p-1} > \binom{t+2}{2}$, which forces $2 {p}^{1/4} > t$. From this we get $|G/B| < 2^{t} p^{t/4}$. Since $t = n/2$, we have $|G/B| < 2^{n/2} p^{n/8} < p^{0.176n}$ (for $p > 1000$). There exists a central subgroup $A = Z(B)$ in $B$ of index at most $60^{n/2}$. So $$k(G) \geq k(B)/|G:B| \geq |A|/|G:B| > |A|/p^{0.176n}.$$ We have $|G| > |V|/(2\sqrt{p-1})$ from which $$|A| \geq |G|/(60^{n/2}p^{0.176n}) > |V|/(60^{n/2} \cdot {p}^{0.176n} \cdot 2\sqrt{p-1}).$$ This gives $k(G) \geq |V|/(60^{n/2} \cdot {p}^{0.352n} \cdot 2\sqrt{p-1})$. But this is larger than $2 \sqrt{p-1}$ for $n \geq 4$.

Now let $121 < p < 1000$. Then it is easy to see that $n(T_{1},V_{1}) \geq 4$. So we have $n(G,V) \geq \binom{t+3}{3}$ by Lemma \ref{l99}. So we may assume that $2\sqrt{p-1} > \binom{t+3}{3} > t^{3}/6$. From this $12^{1/3} p^{1/6} > t$. Since $p < 1000$, we get $t \leq 7$. In fact by looking more closely at the bound using the binomial coefficient, we get $t \leq 5$. Consider the stabilizer $H$ of $V_{1}$. We claim that $H/C_{H}(V_{1})$ has center of order at most $(q-1)/4$. Otherwise $k(H) \geq ((q-1)/4) \cdot k(A_{5}) > q-1$. But then $$k(G) \geq k(H)/5 > (q-1)/5 > 2 \sqrt{q-1}$$ since $q > 121$. But then $T$ has center of size at most $(q-1)/4$. But then going back to the place where we calculated orbits, we see that $n(T_{1},V_{1}) \geq 10$. So $n(G,V) \geq \binom{t+9}{9} \geq \binom{12}{9} = 220 > 64 > 2 \sqrt{p-1}$ (for $t \geq 3$ and $p < 1000$), by Lemma \ref{l99}.
So $t = 2$. We claim that $H/C_{H}(V_{1})$ has center of order at most $(q-1)/8$. Otherwise $k(G) \geq k(H)/2 \geq ((q-1)/8 \cdot k(A_{5}))/2 \geq 2 \sqrt{p-1}$. But then $T$ has center of size at most $(q-1)/8$. But then going back to the place where we calculated orbits, we see that $n(T_{1},V_{1}) \geq 18$. So $n(G,V) \geq \binom{19}{2} = 171 > 64 > 2 \sqrt{p-1}$ (for $p < 1000$), by Lemma \ref{l99}.

So the only remaining cases are: $t \geq 2$, and $p = 59$, $61$, $71$, $79$, $89$, $91$, $101$, $109$, or $119$. If $p=59$ then it can happen that $n(T_{1},V_{1}) = 2$, so in this case we can only say that $t \leq 14$. In all other cases $n(T_{1},V_{1}) \geq 3$, so $t \leq 4$. 

Finally we consider (2/b) of Theorem \ref{orbit}. 

Let $T \leq GL(2,p)$ be as above. $T$ is solvable (and primitive) and $|T/Z(T)| \leq 24$. We may assume that $t \geq 5$. But then, by Lemma \ref{l99}, $n(G,V) \geq \binom{k+4}{5}$ where $k$ is the number of orbits of $T$ on the corresponding module of size $p^2$. Thus we can assume that $\binom{k+4}{5} \leq n(G,V) < 2 \sqrt{p-1}$. Since we may also assume that $k \geq 3$ by order considerations, from this previous inequality we get $113 \leq p$. But then, by Proposition \ref{Seager}, we have $k \geq p/24$. However, since $113 \leq p$, we also have $k \geq \max \{ 5, p/24 \}$. Just by using the bound $k \geq 5$ we can conclude that $p \geq 3970$. But then applying $k \geq p/24$ to the inequality above we get a contradiction.
\end{proof}

\section{Bounding $n(G,V)$ and $k(G)$}

We prove Theorem \ref{main1} by going through the five cases of Theorem \ref{main3}. 

Since both $k(G)$ and $n(G,V)$ are at least $2$, there is nothing to do in cases (1) and (2). Groups in cases (3) and (5) are solvable, so the argument of \cite{HK} applies. Let us assume then that case (4) of Theorem \ref{main3} is satisfied.   

Suppose first that $p=59$ and $t$ is an integer with $2 \leq t \leq 14$. In this case we need $k(G) + n(G,V) -1 \geq 16$. If the center of $T$ has size $58$ then $k(H) > 29 \cdot 5 = 145$ where $H$ is the stabilizer in $G$ of $V_{1}$. So $k(G) > 145/t$. But $n(G,V) \geq t+1$, so $k(G) + n(G,V) -1 > 145/t + t > 16$. (This also works for $t=1$.) So we know that $T$ has at least $3$ orbits on $V_{1}$ by Hering's theorem (see \cite[Chapter XII]{HB}). But then $n(G,V) \geq \binom{t+2}{2}$ by Lemma \ref{l99}. This is at least $16$ unless $t \leq 4$. 

Let $t=4$. Then $n(G,V) \geq 15$ by Lemma \ref{l99}. The result follows from $k(G) \geq 2$.  

In what follows let $C$ be the center of $H/C_{H}(V_{1})$.

Let $t=3$. Then $n(G,V) \geq 10$ by Lemma \ref{l99}. So we would need $k(G) \geq 7$. 
Then $k(G) \geq (k(A_{5}) |C|)/2t = (5/2)|C|/3$. If $|C| \geq 9$ then we are finished. Otherwise the center of $T$ has also at most $8$ elements. Then $n(T_{1},V_{1}) \geq 9$. So $n(G,V) \geq \binom{11}{3} > 22$ by Lemma \ref{l99}.
 
Let $t=2$. Then $n(G,V) \geq 6$ by Lemma \ref{l99}. So we would need $k(G) \geq 11$.  
Then $k(G) \geq (k(A_{5}) |C|)/2t = (5/2)|C|/2$. If $|C| \geq 9$ then we are finished. Otherwise the center of $T$ has also at most $8$ elements. Then $n(T_{1},V_{1}) \geq 9$. So $n(G,V)$ is at least $\binom{10}{2} = 45 > 11$ by Lemma \ref{l99}.

Let $q=61$ and $t=1$. Then $|C| \leq 6$. So $G$ has at least $11$ non-trivial orbits on $V$. So $n(G,V) + k(G) - 1 \geq 11 + k(G) \geq 16 > 2 \sqrt{60}$, a contradiction.

Suppose that $61 < p \leq 119$. Then $k \geq 3$ by Hering's theorem (see \cite[Chapter XII]{HB}). 

Let $t=4$. Then $n(G,V) \geq 15$ by Lemma \ref{l99}. So we would need $k(G) \geq 8$ (since $2\sqrt{118}$ is a bit smaller than $22$). 
Then $k(G) \geq (k(A_{5}) |C|)/2t = (5/2)|C|/4$. If $|C| \geq 13$ then we are finished. Otherwise the center of $T$ has also at most $12$ elements. Then $n(T_{1},V_{1}) \geq 7$. So $n(G,V) \geq \binom{10}{4} > 22$ by Lemma \ref{l99}.   

Let $t=3$. Then $n(G,V) \geq 10$ by Lemma \ref{l99}. So we would need $k(G) \geq 13$.  
Then $k(G) \geq (k(A_{5}) |C|)/2t = (5/2)|C|/3$. If $|C| \geq 16$ then we are finished. Otherwise the center of $T$ has also at most $15$ elements. Then $n(T_{1},V_{1}) \geq 6$. So $n(G,V)$ is at least $\binom{8}{3} = 56 > 22$ by Lemma \ref{l99}. 

Let $t=2$. Then $n(G,V) \geq 6$ by Lemma \ref{l99}. So we would need $k(G) \geq 17$. 
Then $k(G) \geq (k(A_{5}) |C|)/2t = (5/2)|C|/2$. If $|C| \geq 14$ then we are finished. Otherwise the center of $T$ has also at most $13$ elements. Then $n(T_{1},V_{1}) \geq 6$. So $n(G,V)$ is at least $\binom{7}{2} = 21$ by Lemma \ref{l99}. But $k(G) \geq 2$.

This proves Theorem \ref{main1} which in turn establishes Theorem \ref{main}.

\bigskip

{\bf Acknowledgement.} The author thanks Gunter Malle for helpful conversations on this topic and for a careful reading of an earlier version of this paper.

\end{document}